\documentclass[a4paper,12pt]{article}
\usepackage{amssymb,amsmath,amsthm}
\usepackage{hyperref}

\newenvironment{tightlist}{\begin{list}{$\bullet$}
{\setlength{\itemsep}{0ex plus0.2ex}
\setlength{\topsep}{0ex plus0.2ex}}}
{\end{list}}

\DeclareMathOperator{\rk}{rk}

\DeclareMathOperator{\td}{td}  %transcendence degree
\DeclareMathOperator{\ldim}{ldim}  %linear dimension
\DeclareMathOperator{\Qldim}{\ldim_\Q}  %$Q-linear dimension
\DeclareMathOperator{\Jac}{Jac}  % Jacobian
\newcommand{\restrict}[1]{\ensuremath{\!\!\upharpoonright_{#1}}}

\newcommand{\N}{\ensuremath{\mathbb{N}}}
\newcommand{\Z}{\ensuremath{\mathbb{Z}}}
\newcommand{\Q}{\ensuremath{\mathbb{Q}}}
\newcommand{\R}{\ensuremath{\mathbb{R}}}
\newcommand{\C}{\ensuremath{\mathbb{C}}}
\newcommand{\Rexp}{\ensuremath{\mathbb{R}_{\mathrm{exp}}}}
\newcommand{\Cexp}{\ensuremath{\mathbb{C}_{\mathrm{exp}}}}
\newcommand{\Kexp}{\ensuremath{K_{\mathrm{exp}}}}

\newcommand{\graph}{\ensuremath{\mathcal{G}}}

\newcommand{\ga}{\ensuremath{{\mathbb{G}_\mathrm{a}}}}   %additive group of a field
\newcommand{\gm}{\ensuremath{{\mathbb{G}_\mathrm{m}}}}  %mult group of a field
\newcommand{\HV}{\ensuremath{\mathcal{H}_V}} %Set of subgroups / hyperplanes
\newcommand{\ps}[2]{\ensuremath{#1[\![#2]\!]}}
\newcommand{\Laurent}[2]{\ensuremath{#1(\!(#2)\!)}} %Laurent series
\newcommand{\cps}[2][\C]{\ensuremath{#1\{\!\{#2\}\!\}}}  %convergent

\renewcommand{\phi}{\varphi}
\renewcommand{\le}{\ensuremath{\leqslant}}
\renewcommand{\ge}{\ensuremath{\geqslant}}
\newcommand{\tuple}[1]{\ensuremath{\langle #1 \rangle}}
\newcommand{\class}[2]{\ensuremath{\left\{ #1 \,\left|\, #2 \right.\right\}}}

\newcommand{\iso}{\cong}

\newcommand{\into}{\hookrightarrow}
\newcommand{\onto}{\twoheadrightarrow}

\newcommand{\subs}{\subseteq} % diagrams package uses subset for hook
\newcommand{\nstrong}{\ensuremath{\not\kern-4pt\lhd\;}} % nonstrong embedding

\newcommand{\cross}{\ensuremath{\times}}

\newbox\noforkbox \newdimen\forklinewidth
\forklinewidth=0.3pt \setbox0\hbox{$\textstyle\smile$}
\setbox1\hbox to \wd0{\hfil\vrule width \forklinewidth depth-2pt
  height 10pt \hfil}
\wd1=0 cm \setbox\noforkbox\hbox{\lower 2pt\box1\lower
2pt\box0\relax}
\def\unionstick{\mathop{\copy\noforkbox}\limits}

\def\nonfork_#1{\unionstick_{\textstyle #1}}

\setbox0\hbox{$\textstyle\smile$} \setbox1\hbox to \wd0{\hfil{\sl
/\/}\hfil} \setbox2\hbox to \wd0{\hfil\vrule height 10pt depth
-2pt width
                \forklinewidth\hfil}
\newbox\doesforkbox
\setbox\doesforkbox\hbox{\lower 2pt\box1 \lower
2pt\box2\lower2pt\box0\relax}
\def\nunionstick{\mathop{\copy\doesforkbox}\limits}

\def\fork_#1{\nunionstick_{\textstyle #1}}

\newcommand{\ra}[3]{\ensuremath{#1 \stackrel{#2}{\longrightarrow} #3}}
\newcommand{\leteq}{\mathrel{\mathop:}=}

\DeclareMathOperator{\Kldim}{K\mbox{-}ldim}
\DeclareMathOperator{\logD}{lD}

\newtheorem{prop}{Proposition}[section]

\newtheorem{theorem}[prop]{Theorem}

\newtheorem{conj}[prop]{Conjecture}

\theoremstyle{definition}
\newtheorem{defn}[prop]{Definition}

\title{Variants of Schanuel's conjecture}
\author{Jonathan Kirby}
\date{Version 2.0, \today}

\begin{document}

\maketitle

\subsection*{Introduction, January 2018}
This is a collection of variants of Schanuel's conjecture and the known
dependencies between them. It was originally written in 2007, and made available for a time on my webpage. I have been asked by a few people to make it available again and have taken the opportunity to make some minor revisions now in January 2018. The treatment is far from exhaustive of the literature, and for the most part consists of statements which were under discussion in the Oxford logic group between around 2002 and 2007. Many of these appeared in Boris Zilber's papers \cite{Zilber00fpe}, \cite{Zilber02esesc}, and \cite{Zilber03powers}. The general idea (although not much stressed in this article) is that Schanuel-type statements are obtained by counting the degrees of freedom in a system of equations, and that if this is negative then the system has no solutions.

Some parts of this article are now obsolete. The section on generic powers has been superseded by the paper \cite{BKW10}. Conjecture~\ref{SSC} is false but the issue of what part of Schanuel's conjecture is first-order expressible is discussed in \cite{ECFCIT}. The relationship between non-standard integer powers and the CIT is explained in chapter~6 of \cite{Bays_thesis}.

\subsection*{Conventions}
We adopt the convention that theorems are unconditionally proved
statements (about numbers, fields, exponential maps, etc,) and
propositions are (unconditionally proved) relationships between
conjectures. Most proofs are omitted. Algebraic varieties defined over
a subfield of the complex numbers $\C$ are identified with their
$\C$-points. Throughout the article we write $\td_\Q(x_1,\ldots,x_n)$
to mean the transcendence degree of the field extension
$\Q(x_1,\ldots,x_n)/\Q$, and $\Qldim(x_1,\ldots,x_n)$ to mean the
$\Q$-linear dimension of the $\Q$-vector space spanned by
$x_1,\ldots,x_n$.

\newpage

\tableofcontents

\section{Statements of Schanuel's conjecture}

We begin with Stephen Schanuel's original conjecture, in several
equivalent formulations.

\begin{conj}[Schanuel's conjecture (SC)]
  Let $a_1,\ldots,a_n$ be $\Q$-linearly independent complex
  numbers. Then $\td_\Q(a_1,e^{a_1},\ldots,a_n,e^{a_n}) \ge n$.
\end{conj}

\begin{conj}[(SC), version 2]
  Let $a_1,\ldots,a_n$ be complex numbers and suppose that
  $\td_\Q(a_1,e^{a_1},\ldots,a_n,e^{a_n}) < n$. Then there are
  integers $m_1,\ldots,m_n$, not all zero, such that $\sum_{i=1}^n
  m_ia_i = 0$.
\end{conj}

We define a ``predimension function'' $\delta$ on $n$-tuples of
complex numbers by
\[\delta(a_1,\ldots,a_n) \leteq \td(a_1,e^{a_1},\ldots,a_n,e^{a_n}) -
\Qldim(a_1,\ldots,a_n).\] 
\begin{conj}[(SC), version 3]\label{predim version}
  Let $a_1,\ldots,a_n$ be complex numbers. Then
  $\delta(a_1,\ldots,a_n) \ge 0$.
\end{conj} 

The predimension  function can be relativised. For a subset $A \subs \C$, we can define 
\[\delta(a_1,\ldots,a_n/A) \leteq \td(a_1,e^{a_1},\ldots,a_n,e^{a_n}/A,\exp(A)) -
\Qldim(a_1,\ldots,a_n/A).\] 
For example, we can take $A = \ker(\exp) = 2 \pi i \Z$, and that gives a version of Schanuel's conjecture ``over the kernel''.

\begin{conj}[SC over the kernel]
 Let $a_1,\ldots,a_n$ be complex numbers. Then
  $\delta(a_1,\ldots,a_n/\ker(\exp)) \ge 0$.
\end{conj}

There is a weaker version of (SC) which also ignores the kernel.
\begin{conj}[(Weak SC)]
  Let $a_1,\ldots,a_n$ be complex numbers and suppose that
  $\td_\Q(a_1,e^{a_1},\ldots,a_n,e^{a_n}) < n$. Then there are
  integers $m_1,\ldots,m_n$, not all zero, such that $\prod_{i=1}^n
  e^{m_ia_i} = 1$.
\end{conj}

\begin{prop}
  (SC) $\implies$ (SC over the kernel) $\implies$ (Weak SC).
\end{prop}
%\begin{proof}
%  The multiplicative dependence between the exponents of the $a_i$
%  follows from the linear dependence between the $a_i$ themselves,
%  because the exponential function is a homomorphism from the additive
%  group to the multiplicative group.
%\end{proof}
The converse implication (Weak SC) $\implies$ (SC over the kernel) is false. For example, if $e$ and $\pi$ were algebraically dependent, but Schanuel's conjecture had no other counterexamples, then (Weak SC) would be true but $\delta(1/\ker) = -1$.

The implication (SC over the kernel) $\implies$ (SC) is true, because $2\pi i$ is transcendental, and the kernel is a cyclic group. However the equivalent implication is not necessarily true for other exponential fields.

\subsection*{Geometric statements}

In a different direction, we can use the fact that the transcendence
degree of an $m$-tuple of complex numbers is smaller than $n$ iff the
tuple lies in an algebraic variety of dimension less than $n$.
\begin{conj}[(SC), version 4]\label{variety version}
  Let $a_1,\ldots,a_n$ be complex numbers, and suppose that the
  $2n$-tuple $(a_1,e^{a_1},\ldots, a_n, e^{a_n})$ lies in an algebraic
  subvariety $V$ of $\C^{2n}$ which is defined over $\Q$ and of
  dimension strictly less than $n$. Then there are integers
  $m_1,\ldots,m_n$, not all zero, such that $\sum_{i=1}^n m_ia_i = 0$.
\end{conj}

Continuing in this direction, we can find a more geometric statement.
Write $\gm(\C)$ for the multiplicative group of the complex numbers,
and $\ga(\C)$ for its additive group. $\ga(\C)$ is just $\C$, and is
naturally identified with the universal covering space of $\gm(\C)$,
and the exponential map is the covering map
$\ra{\ga(\C)}{\exp}{\gm(\C)}$.

$\ga(\C)$ is also naturally identified with tangent space of $\gm(\C)$
at the identity. This is the Lie algebra of $\gm(\C)$, written
$L\gm(\C)$. The tangent bundle $T\gm(\C)$ is naturally isomorphic to
$L\gm(\C) \cross \gm(\C)$. Thus the graph of the exponential map is an
analytic subvariety of the tangent bundle $\graph \subs T\gm(\C)$.
% We also use the natural identification of $T(\gm(\C)^n)$ with
% $(T\gm(\C))^n$, and drop the brackets.
\begin{conj}[(SC), version 5]\label{SC intersection version}
  Let $V$ be an algebraic subvariety of $T\gm(\C)^n$, defined over
  $\Q$, of dimension strictly less than $n$. Then the intersection
  $\graph \cap V$ is contained in the union 
  \[\bigcup \class{TH}{H \mbox{ is a proper algebraic subgroup of
    }\gm(\C)^n}.\]
\end{conj}

\subsection*{Power series version}

For any ring $R$, the ring of formal power series over $R$ is given by
\[\ps{R}{t} = \class{\sum_{n\in \N}a_n t^n}{a_n \in R}\]
with the usual term-by-term addition and multiplication. If $f$ and
$g$ are two power series and $g$ has no constant term, then there is a
formal composite $f(g)$.  The exponential function is represented by
the formal power series $\sum_{n\in \N} \frac{t^n}{n!}$, and so the
ring $\ps{R}{t}$ admits a partial exponential map defined on the
subset $t\ps{R}{t}$, that is, on the principal ideal generated by $t$.
In the case where $R$ is an exponential ring, in particular for
$\Cexp$, the exponential map can be extended to the whole of
$\ps{R}{t}$. For $a = \sum_{n\in \N}a_n t^n$, define $\exp(a) =
\exp(a_0)\exp(a-a_0)$, where the first $\exp$ is that defined on $R$
and the second is that defined on power series with no constant term.
We can also consider power series in several variables
$t_1,\ldots,t_r$. There are the usual derivations
$\frac{\partial}{\partial t_j}$ for $j=1,\ldots,r$. The \emph{Jacobian
  matrix} of a tuple $f = (f_1,\ldots,f_n)$ is $\Jac(f) =
\left(\frac{\partial f_i}{\partial t_j}\right)_{i,j}$. The rank of
this matrix is used.

The analogue of Schanuel's conjecture for power series is as follows.
The special case of this statement when $r=0$ is just (SC). In theorem
2 of \cite{Ax71}, James Ax proved the converse implication.
\begin{conj}[(SC), power series version]
  Let $f_1,\ldots,f_n \in \ps{\C}{t_1,\ldots,t_r}$. If
  $\td_\C(f_1,\exp(f_1),\ldots,f_n,\exp(f_n)) - \rk\Jac(f) < n$ then
  there are $m_i \in \Z$, not all zero, such that $\sum_{i=1}^n m_if_i
  = 0$.
\end{conj}

We can also consider the ring of convergent power
series $\cps{t}$, the subring of $\ps{\C}{t}$ consisting of those
power series with non-zero radius of convergence. $\cps{t}$ is
naturally a (total) exponential ring. The restriction of the above
power series version of (SC) to convergent power series is
intermediate between (SC) and the power series statement, hence it is
equivalent to both.

\subsection*{Roy's version}

Damien Roy \cite{Roy01} has shown that (SC) is equivalent to the
following statement, which is in the style of transcendental number
theoretic statements saying that transcendence is equivalent to having
a good rational approximation.
\begin{conj}
  Let $n \in \N$, and fix positive real numbers $s_0,s_1,t_0,t_1,u$
  such that $\max\{1,t_0,2t_1\} < \min\{s_0,2s_1\}$ and
  \[\max\{s_0,s_1+t_1\} < u < \frac{1}{2}(1+t_0+t_1).\]

  Then for all $x_1,\ldots,x_n \in \C$ and all
  $\alpha_1,\ldots,\alpha_n \in \C^\times$, either 
  \begin{tightlist}
  \item the $x_i$ are $\Q$-linearly dependent, or 
  \item $\td_\Q(\bar{x},\bar{\alpha}) \ge n$, or
  \item there is $N_0 \in \N$ such that for all $N \in \N$ greater
    than $N_0$ there is a nonzero polynomial $P_N \in \Z[X_0,X_1]$
    with partial degrees at most $N^{t_0}$ in $X_0$ and $N^{t_1}$ in
    $X_1$, and all coefficients of modulus at most $e^N$ such that for
    all $k,m_1,\ldots,m_n \in \N$ with $k \le N^{s_0}$ and
    $\max\{m_1,\ldots,m_n\} \le N^{s_1}$ we have
    \[\left|(\mathcal{D}^k P_N)\left(\sum_{j=1}^n m_j x_j, \prod_{j=1}^n
        \alpha_j^{m_j}\right)\right| \le e^{-N^{u}}\]
  \end{tightlist}
  where $\mathcal{D}$ is the derivation $\frac{\partial}{\partial X_0}
  + X_1 \frac{\partial}{\partial X_1}$.
\end{conj}

\section{Known cases}

The following are known special cases of Schanuel's conjecture.

\begin{description}
\item[Hermite, 1873]
  $e$ is transcendental. (Special case of (SC) with $n = 1, a_1 = 1$.)

\item[Lindemann, 1882]
  $\pi$ is transcendental. ($n = 1, a_1 = \pi i$.)

\item[Lindemann, 1882]
  If $a$ is algebraic then $e^a$ is transcendental. ($n=1$)

\item[Lindemann-Weierstrass theorem, 1885]
  If $a_1,\ldots,a_n$ are $\Q$-linearly independent algebraic numbers
  then $e^{a_1},\ldots,e^{a_n}$ are algebraically independent. (All
  the $a_i$ are algebraic.)

\item[Nesterenko, 1996]
  $\pi$ and $e^\pi$ are algebraically independent. ($n=2, a_1 = \pi,
  a_2 = i\pi$.)
\end{description}

However, some very simple cases are unknown. For example, it is not
known if $e$ and $\pi$ are algebraically independent. ($n=2, a_1 = 1,
a_2 = \pi i$.)

There are other known theorems which give special cases, such as the Gelfond-Schneider theorem and Baker's theorem.

\section{Uniform SC and Strong SC}

Schanuel's conjecture was strengthened to a uniform version by Zilber in \cite{Zilber02esesc}.
\begin{conj}[(USC), Uniform Schanuel Conjecture]\label{USC}
  For each algebraic subvariety $V \subs T\gm(\C)^n$ defined over $\Q$
  and of dimension strictly less than $n$, there is a finite set $\HV$
  of proper algebraic subgroups of $\gm^n$ such that for any
  $a_1,\ldots,a_n \in \C$, if $(a_1,e^{a_1},\ldots, a_n, e^{a_n}) \in
  V$ then there is $H \in \HV$ such that $(a_1,\ldots,a_n) \in LH + 2\pi i \Z^n$, and either $H$ has codimension at least 2, or $(a_1,\ldots,a_n) \in LH$.
\end{conj}

This conjecture can be written in an intersection version similar to Conjecture~\ref{SC intersection version}, but it is a little more messy.

It is easy to see from the statements that (USC) implies (SC). The
converse is not known, but the diophantine conjecture CIT (see later)
implies that any exponential field satisfying (SC) does indeed satisfy
(USC).  However, we can do without CIT in some cases. Write
(SC${}_\R$) for the conjecture that the real exponential field $\Rexp$
satisfies (SC), that is that the statement of Schanuel's conjecture
holds when the $a_i$ are real numbers. Similarly, write (USC${}_\R$)
for the conjecture that $\Rexp$ satisfies (USC).
\begin{prop}[\cite{USCR}]
  (SC${}_\R$) $\iff$ (USC${}_\R$). Moreover, if an exponential field (or
  partial exponential field) is o-minimal with analytic
  cell-decomposition and satisfies (SC) then it also satisfies (USC).
\end{prop}
The structure of $\Rexp$ with restricted sine interprets the complex
field with complex exponentiation defined on the strip $\R \cross
[-\pi i,\pi i]$. This structure is known to be o-minimal, with
analytic cell-decomposition.
\begin{proof}
  The proof of this for $\Rexp$ is in \cite{USCR}, but it only uses
  the fact that $\Rexp$ is o-minimal and has analytic
  cell-decomposition, so the same proof gives the full statement.
\end{proof}

Another strengthening of (SC) is the following.
\begin{conj}[(SSC), Strong Schanuel conjecture]\label{SSC}
  The statement of (SC) holds for all ultrapowers of $\Cexp$
  (including $\Cexp$ itself). Equivalently, (SC) is true of $\Cexp$
  and is part of its first order theory.
\end{conj}

{\it Added in 2018: The conjecture (SSC) is false. See \cite{ECFCIT} for a discussion of the issue.}

\section{Parametric SC}

(SC) version 4 applies to subvarieties defined over $\Q$, but there is
a statement dealing with subvarieties defined over any subfield of
$\C$.
\begin{conj}[(SC${}_{param}$), Schanuel conjecture with parameters]
  Let $V \subs T\gm(\C))^n$ be any algebraic subvariety, of dimension
  strictly less than $n$. Then there is $l \in \N$ and $b_1,\ldots,b_l
  \in \C$ such that if $a_1,\ldots,a_n \in \C$ such that
  $(a_1,e^{a_1},\ldots,a_n,e^{a_n}) \in V$ then there are
  $m_1,\ldots,m_{n+l} \in \Z$, with $m_1,\ldots,m_n$ not all zero,
  such that $\sum_{i=1}^n m_ia_i + \sum_{i=1}^l m_{i+n}b_i = 0$.

  Furthermore, if $V$ is a fibre $W(\bar{p})$ where $W \subs
  T\gm(\C)^n \cross P$ and $P$ is an algebraic subvariety of $\C^k$
  defined over $\Q$ then we may take $l = k(n+1)$.
\end{conj}

\begin{prop}[{\cite[Proposition 4]{Zilber02esesc}}]
  (SC) $\implies$ (SC${}_{param}$)
\end{prop}

The uniform version is simpler to state.

\begin{conj}[(USC${}_{param}$), Uniform Schanuel conjecture with
  parameters] \ \\
  Let $V \subs T\gm(\C))^n$ be any algebraic subvariety, of dimension
  strictly less than $n$. Then there is a finite collection
  $\mathcal{K}_V$ of cosets of the form $g \cdot TH$ where $H$ is a
  proper algebraic subgroup of $\gm(\C)^n$ such that $\graph \cap V
  \subs \bigcup \mathcal{K}_V$.
\end{conj}

\section{Ax's theorem and some consequences}

An important piece of work on Schanuel's conjecture was done by James
Ax, in \cite{Ax71}, using differential fields. Let $F$ be a field (of
characteristic zero), $\Delta$ a set of derivations on $F$, and $C$
the common field of constants, which is given by $C = \bigcap_{D \in
  \Delta} \class{x \in F}{Dx = 0}$. Usually $\Delta$ will be a finite
set or a finite-dimensional $F$-vector space of derivations. If $x =
(x_1,\ldots,x_n)$ is a finite tuple of elements of $F$, the
\emph{Jacobian matrix} of $x$ is defined to be the matrix $\Jac(x)$
with entries $(Dx_i)$ for $i=1,\ldots,n$ and $D \in \Delta$. The order
of the columns will not matter, nor that the matrix is infinite if
$\Delta$ is infinite, because we are only concerned with the rank of
the matrix, $\rk \Jac(x)$, which is at most $n$.
Write
\[\Gamma = \class{(x,y) \in T\gm(F)}{\frac{Dy}{y} = Dx \mbox{ for each
  }D \in \Delta}\] for the solution set to the exponential
differential equation. In this language, Ax's main theorem is as
follows.
\begin{theorem}[(SC${}_D$), Differential field Schanuel condition,
  Ax's theorem]
  If $x_1,y_1,\ldots,x_n,y_n \in \Gamma^n$ and
  $\td_C(x_1,y_1,\ldots,x_n,y_n) - \rk\Jac(x) < n$ then there are $m_i
  \in \Z$, not all zero, such that $\sum_{i=1}^n m_ix_i \in C$ and
  $\prod_{i=1}^n y_i^{m_i} \in C$.
\end{theorem}

An equivalent statement is the power series version of (SC),
restricted to power series with no constant term.
\begin{theorem}[{\cite[Corollary 1]{Ax71}}]\label{Ax formal ps}
  Let $C$ be a field of characteristic zero and let $f_1,\ldots,f_n
  \in \ps{C}{t_1,\ldots,t_r}$ be power series with no constant term.
  If $\td_C(f_1,\exp(f_1),\ldots,f_n,\exp(f_n)) - \rk\Jac(f) < n$ then
  there are $m_i \in \Z$, not all zero, such that $\sum_{i=1}^n m_if_i
  = 0$.
\end{theorem}

An immediate corollary translates this into a result about analytic
functions.
\begin{theorem}[{\cite[Corollary 2]{Ax71}}]
  Let $C$ be an algebraically closed field of characteristic zero
  which is complete with respect to a non-discrete absolute value. Let
  $f_1,\ldots,f_n$ be analytic functions in some polydisc about the
  origin $0$ in $C^r$ such that each $\exp(f_i)$ is defined and for
  which the $f_i - f_i(0)$ are $\Q$-linearly independent. Then
  \[\td_C(f_1,\exp(f_1),\ldots,f_n,\exp(f_n)) - \rk\Jac(f) \ge n.\]
\end{theorem}

Robert Coleman strengthened theorem~\ref{Ax formal ps} in
\cite{Coleman80}, in the special case that the power series are
defined over $\bar{\Q}$. For a field $F$, write $\Laurent{F}{t}$ for
the field of Laurent series, that is, the field of fractions of
$\ps{F}{t}$. Let $\mathcal{O}$ be the ring of algebraic integers (a
subring of $\bar{\Q}$) and let $k$ be its field of fractions.
\begin{theorem}
  Let $f_1,\ldots,f_n \in \ps{\bar{\Q}}{t_1,\ldots,t_r}$ be power
  series with no constant term, and suppose they are $\Q$-linearly
  independent. Then
  \[\td_{\Laurent{k}{t_1,\ldots,t_r}}(f_1,\exp(f_1),\ldots,f_n,\exp(f_n))
  \ge n.\]
\end{theorem}

The uniform version of Ax's theorem is also a theorem.
\begin{theorem}[(USC${}_D$), \cite{DPhil}]
  For each parametric family $(V_c)_{c \in P(C)}$ of subvarieties of
  $T\gm^nS$, with $V_c$ defined over $\Q(c)$, there is a finite set
  $\HV$ of proper algebraic subgroups of $\gm^n$ such that for each $c
  \in P(C)$ and each $(x,y) \in \Gamma^n \cap V_c$, if $\dim V_c -
  \rk\Jac(x,y) < n$, then there is $\gamma \in TS(C)$ and $H \in \HV$
  such that $(x,y)$ lies in the coset $\gamma \cdot TH$.
\end{theorem}

A corollary deals with complex analytic varieties.
\begin{theorem}[{\cite[Theorem 8.1]{DPhil}}]
  Let $P$ be an algebraic variety and $(V_p)_{p \in P(\C)}$ be a
  parametric family of algebraic subvarieties of $T\gm(\C)^n$. There
  is a finite collection $\HV$ of proper algebraic subgroups of
  $\gm^n$ with the following property:

  If $p \in P$ and $W$ is a connected component of the analytic
  variety $\mathcal{G} \cap V_p$ with analytic dimension $\dim W$
  satisfying $\dim W > \dim V_p - n$, then there is $H \in \HV$ and $g
  \in T\gm(\C)^n$ such that $W$ is contained in the coset $g\cdot TH$.
\end{theorem}

\section{Two-sorted exponentiation}

Consider the two-sorted structure of the exponential map with the
domain and codomain as separate copies of $\C$, with the full field
structure on each. The difference with the one-sorted case is that we
do not have a chosen isomorphism between the two sorts. In this
setting, we cannot ask about algebraic relations between elements
$a_i$ of the domain and their images $e^{a_i}$ in the codomain. Thus
the relevant version of Schanuel's conjecture in this setting is
weaker. \cite{Zilber00fpe}

\begin{conj}[(2-sorted SC)]
  Let $a_1,\ldots,a_n$ be complex numbers. Then
  \[\td_\Q(a_1,\ldots,a_n) + \td_\Q(e^{a_1},\ldots,e^{a_n}) -
  \Qldim(a_1,\ldots,a_n) \ge 0.\]
\end{conj} 

\begin{prop}
  (SC) $\implies$ (2-sorted SC)
\end{prop}
\begin{proof}
  $\td_\Q(a_1,\ldots,a_n) + \td_\Q(e^{a_1},\ldots,e^{a_n}) \ge
  \td_\Q(a_1,e^{a_1},\ldots,a_n,e^{a_n})$.
\end{proof}

\section{Raising to powers}

The exponential function is used to define the (multivalued) functions
of raising to a power. For $a,b,\rho \in \C$, write $b=a^\rho$ to mean
that there is $\alpha \in \C$ such that $e^\alpha = a$ and
$e^{\rho\alpha} = b$. Schanuel's conjecture (for exponentiation)
implies statements about these ``raising to powers'' functions. We say
that $x_1,\ldots,x_n$ are \emph{multiplicatively dependent} iff there
are $m_1,\ldots,m_n \in \Z$, not all zero, such that $\prod_{i=1}^n
x_i^{m_i} = 1$. Otherwise they are \emph{multiplicatively
  independent}.

\subsection*{Algebraic powers}

\begin{conj}[(SCP${}_{alg}$), Schanuel Conjecture for raising to
  an algebraic power]
  Let $\rho \in \bar{\Q}$, and $a_1,b_1,\ldots,a_n,b_n \in \C$ such
  that $b_i = a_i^{\rho}$. If we have $\td_\Q(a_1,b_1,\ldots,a_n,b_n) < n$ then
  there are $m_1,\ldots,m_{2n} \in \Z$, not all zero, such that
  $\prod_{i=1}^n a_i^{m_i} b_i^{m_{n+i}} = 1$.

  Equivalently, if the $a_i$ and $b_i$ are multiplicatively
  independent then we have $\td_\Q(a_1,b_1,\ldots,a_n,b_n) \ge n$.
\end{conj}

If $\rho$ is rational then this statement is uninteresting.
\begin{prop}
  (Weak SC) $\implies$ (SCP${}_{alg}$)
\end{prop}
\begin{proof}
  For each $i$ let $\alpha_i$ be such that $e^{\alpha_i} = a_i$ and
  $e^{\rho\alpha_i} = b_i$. If we have $\td_\Q(a_1,b_1,\ldots,a_n,b_n)
  < n$ then
  $\td_\Q(\alpha_1,a_1,\ldots,\alpha_n,a_n,\rho\alpha_1,b_1,\rho\alpha_n,b_n)
  < 2n$, so by (Weak SC) there are $m_1,\ldots,m_{2n} \in \Z$, not all
  zero, such that we have $\prod_{i=1}^n a_i^{m_i} b_i^{m_{n+i}} = 1$, as
  required.
\end{proof}

The $n=1$ case is known.
\begin{theorem}[Gelfond-Schneider theorem, 1934]
  If $a,\rho$ are algebraic, $b = a^\rho$, $a \neq 0,1$, and $\rho$ is
  irrational, then $b$ is transcendental. (The special case of
  (SCP${}_{alg}$) with $n=1$.)
\end{theorem}

\subsection*{Generic powers}

{\it See the paper \cite{BKW10} for a proof of the conjecture discussed in this section, in a much improved and generalised form.}

There is no ``correct'' simple form of a conjecture for raising to any
power. For example, let $r = \log(3) / \log(2)$ in $\R$. Then $2^r =
3$, but $\td_\Q(2,3) = 0$. However, there is a simple form of the
statement when the power is generic, that is, not exponentially
algebraic.

\begin{defn}[\cite{Macintyre96}]
  Let $K_{\exp} = \tuple{K;+,\cdot,\exp}$ be an exponential field, and
  let $B \subs K$. An element $\alpha$ of $K$ is said to be
  \emph{exponentially algebraic} over $B$ in $K_{\exp}$ iff there are
  $n,m \in \N$, tuples $a = (a_1,\ldots,a_n) \in K^n$, $b =
  (b_1,\ldots,b_m) \in B^m$, and   $f_1,\ldots,f_n$ in the ring of exponential polynomials $\Z[x_1,\ldots,x_n,y_1,\ldots,y_m]^E$ such that:
  \begin{tightlist}
    \item[i)] $a_1 = \alpha$,
    \item[ii)] For each $i=1,\ldots,n$, $f_i(a,b)=0$, and
    \item[iii)] $ \begin{vmatrix} \frac{\partial f_1}{\partial x_1} &
        \cdots &\frac{\partial
          f_1}{\partial x_n}\\
        \vdots & \ddots & \vdots \\
        \frac{\partial f_n}{\partial x_1} & \cdots &\frac{\partial
          f_n}{\partial x_n} \end{vmatrix} (a,b) \neq 0$
  \end{tightlist}
  When $B = \emptyset$ we say that $\alpha$ is \emph{exponentially
    algebraic} in $K_{\exp}$.
\end{defn}

Warning: The notion of exponential algebraicity depends on the
exponential field. For example, assuming Schanuel's conjecture there are 
real numbers which are exponentially algebraic in $\Cexp$ but are not
exponentially algebraic in $\Rexp$. Schanuel's conjecture even implies that
$\pi$ is such a number.

\begin{conj}[(SCP${}_{gen}$), Schanuel Conjecture for raising to
  a generic power]
  Let $\rho \in\C$, not exponentially algebraic in $\Cexp$, and let
  $a_1,b_1,\ldots,a_n,b_n \in \C$ such that $b_i = a_i^{\rho}$. If the
  $a_i$ and $b_i$ are multiplicatively independent then
  $\td_\Q(a_1,b_1,\ldots,a_n,b_n) \ge n$.
\end{conj}

\begin{prop}
  (SC) $\implies$ (SCP${}_{gen}$)
\end{prop}
\begin{proof}[Sketch proof]
  Let $\alpha_i$ be such that $e^{\alpha_i} = a_i$. Suppose
  $\td_\Q(a,b) < n$. Then $\td_\Q(\alpha,\rho,a,b) < 2n + 1$. So
  $\delta(\alpha,\rho\alpha) < 2n + 1 - \Qldim(\alpha,\rho \alpha)$.

  But $\delta(\alpha,\rho\alpha) \ge 1$, using (SC) and the fact that
  $\rho$ is not exponentially algebraic, so $\Qldim(\alpha,\rho \alpha)
  < 2n$. Thus $a,b$ are multiplicatively dependent.
\end{proof}

Alex Wilkie proved two results about raising to generic real powers.
\begin{theorem}[\cite{Wilkie_SCremark1}]
  Let $r$ be a real number not definable in $\Rexp$ (equivalently, $r$
  is not exponentially algebraic in $\Rexp$), and let $a_1,\ldots,a_n$
  be positive real numbers. If the set
  $\{a_1,a_1^r,\ldots,a_n,a_n^r\}$ is multiplicatively independent
  then $\td_\Q(a_1,a_1^r,\ldots,a_n,a_n^r) \ge n$.
\end{theorem}

\begin{theorem}[\cite{Wilkie_SCremark2}]
  Let $\mathcal{R}$ be the structure
  $\tuple{\R;+,\cdot,\exp,\sin\restrict{[-\pi,\pi]}}$, and let $r$ be
  a real number not definable in $\mathcal{R}$. (Equivalently, $r$ is
  not exponentially algebraic in $\Cexp$.) If $a_1,\ldots,a_n$ are
  multiplicatively independent real numbers then
  $\td_\Q(a_1^{1+ir},\ldots,a_n^{1+ir}) \ge n/2$.
\end{theorem}
These statements were obtained using Ax's theorem (SC)${}_D$. Uniform
versions can be obtained by using (USC)${}_D$, and this is done in the
second case in \cite{Wilkie_SCremark2}.

\subsection*{Non-standard integer powers}
{\it See chapter~6 of Martin Bays' thesis \cite{Bays_thesis} for a better account of non-standard integer powers.}

There is a notion of raising to a non-standard integer power. In
$\Cexp$, a number $z$ lies in $\Z$ iff it satisfies the formula
\[\phi(z) \equiv \forall y[e^y = 1 \to e^{zy} = 1].\]
In any exponential field, the map $x \mapsto x^\rho$ is a single-valued
function precisely when $\phi(\rho)$ holds.

A \emph{non-standard integer} in an exponential field $\Kexp$ with a
non-trivial kernel is an element $\rho \in \Kexp$ such that $\Kexp
\models \phi(\rho)$, but which is not in $\Z$ (the standard integers).
There are elementary extensions of $\Cexp$ (for example, ultrapowers),
in where there are non-standard integers.

For an exponential field $\Kexp$ with a non-standard integer $\rho$, we
isolate the following property.
\begin{description}
\item[(SCP${}_{nsi}$), SC for raising to a non-standard integer power].\\
  If $a_1,\ldots,a_n \in \Kexp$ are such that the $a_i$ and $a_i^\rho$
  are multiplicatively independent then
  $\td_\Q(a_1,a_1^\rho,\ldots,a_n,a_n^\rho) \ge n$.
\end{description}
\begin{conj}
  (SCP${}_{nsi}$) holds for all non-standard integers in all
  ultrapowers of $\Cexp$.
\end{conj}

Note that if $k$ is a nonzero kernel element and $\rho$ is a
non-standard integer then $\exp(k)=1$ shows
that $k$ is exponentially algebraic, and then $\exp(\rho k) = 1$ shows
that $\rho$ is exponentially algebraic. So (SCP${}_{nsi}$) and
(SCP${}_{gen}$) are independent statements.

\subsection*{A field of powers}

It is natural to consider not just a single (multivalued) power function, but a field of powers. This was done in Zilber's paper \cite{Zilber03powers}.
This setup can be captured by another
two-sorted model, $\ra{\C}{\exp}{\C}$, in this case where the domain
sort has the structure of a $K$-vector space for some field $K$, not
the full field structure. We take $K$ to be a subfield of $\C$ of
finite transcendence degree $d$, and the domain to have the natural
$K$-vector space structure. The Schanuel conjecture for this setting
is as follows. 
\begin{conj}[(SCP${}_K$), SC for raising to powers from the field $K$]
  Let $a_1,\ldots,a_n$ be complex numbers. Then
  \[\Kldim(a_1,\ldots,a_n)  + \td_\Q(e^{a_1},\ldots,e^{a_n}) -
  \Qldim(a_1,\ldots,a_n) \ge -d.\]
\end{conj}

\begin{prop}
  (2-sorted SC) $\implies$ (SCP${}_K$)
\end{prop}
\begin{proof}
\begin{eqnarray*}
  \Kldim(a_1,\ldots,a_n) + d & \ge & \td_K(a_1,\ldots,a_n) + d \\
  & \ge & \td_K(a_1,\ldots,a_n) + \td_\Q(K) \\
  & \ge & \td_\Q(a_1,\ldots,a_n)
\end{eqnarray*}
\end{proof}

\begin{prop}
  (SCP${}_{\bar{\Q}}$) $\implies$ (SCP${}_{alg}$)
\end{prop}
\begin{proof}
  Let $\rho \in \bar{\Q}$, and $a_1,b_1,\ldots,a_n,b_n \in \C$ such
  that $b_i = a_i^{\rho}$, and suppose the $a_i$ and $b_i$ are
  multiplicatively independent. For each $i$, choose $\alpha_i \in \C$
  such that $\exp(\alpha_i) = a_i$. Then the $\alpha_i$ and
  $\rho\alpha_i$ are $\Q$-linearly independent. Since $K = \Q(\rho)$,
  we have $d = 0$. Thus from (SCP${}_{\bar{\Q}}$) we get
  \[\Kldim(\bar{\alpha},\rho \bar{\alpha}) + \td_\Q(\bar{a},\bar{b}) -
    \Qldim(\bar{\alpha},\rho \bar{\alpha}) \ge 0.\] 
   
  The $\alpha_i$ and $\rho\alpha_i$ are $\Q$-linearly independent, and
  $\Kldim(\bar{\alpha},\rho \bar{\alpha}) \le n$, so this becomes
  \[\td_\Q(\bar{a},\bar{b}) - 2n \ge - n\]
  which gives the result.
\end{proof}

However, if $\rho$ is a non-algebraic power then, taking $K =
\Q(\rho)$ we have $d=1$, and so (SCP${}_{\Q(\rho)}$) does not imply
(SCP${}_{gen}$), nor (SCP${}_{nsi}$).

\section[The CIT]{The Conjecture on Intersections of Tori \\ with
  algebraic varieties}

{\it Added in 2018: The CIT is now commonly known as the multiplicative case of the Zilber-Pink conjecture.}

Boris Zilber gave in \cite{Zilber02esesc} a conjecture on
intersections of tori with algebraic varieties (CIT) which is
connected to the uniformity in Schanuel conditions. It has a purely
algebraic (Diophantine) statement, as opposed to the statements above
which involve some sort of analytic exponential map. Surprisingly, it
can also be seen as a Schanuel condition itself.

As before, we identify algebraic varieties with their complex points.
\begin{defn}
  If $U$ is a smooth, irreducible algebraic variety, $V,W$ are
  irreducible subvarieties of $U$ and $X$ is an irreducible component
  of $V \cap W$ then $X$ is said to be a \emph{typical component} of
  the intersection iff $\dim X = \dim V + \dim W - \dim U$ and it is
  said to be an \emph{atypical component} iff $\dim X > \dim V + \dim
  W - \dim U$.
\end{defn}
Since $U$ is smooth, and we are considering the points in an
algebraically closed field, it is impossible to have $\dim X < \dim V
+ \dim W - \dim U$.

\begin{conj}[(CIT)]
  For any algebraic subvariety $V$ of $G = \gm^n(\C)$, defined over
  $\bar{\Q}$, there is a finite collection $\HV$ of proper algebraic
  subgroups of $G$ such that for any algebraic subgroup $T$ of $G$, if
  $X$ is an atypical component of $V \cap T$ in $G$ then there is $H
  \in \HV$ such that $X \subs H$.
\end{conj}

There is a version for parametric families of subvarieties as well.
\begin{conj}[(CIT${}_{param}$), CIT with parameters]
  Let $P$ be a complex algebraic variety and let $(V_p)_{p \in P}$ be
  a parametric family of subvarieties of $G = \gm^n(\C)$, the family
  defined over $\bar{\Q}$.

  Then there is a finite collection $\HV$ of proper algebraic
  subgroups of $G$, a natural number $t = t_V$ and functions $c_i:
  P\cross G \to G$ for $i=1,\ldots,t$ such that for any coset $T\cdot
  g$ of any algebraic subgroup of $G$, if $X$ is an atypical component
  of $V_p \cap T\cdot g$ in $G$ then there is $H \in \HV$ and $i \in
  \{1,\ldots,t\}$ such that $X \subs H\cdot c_i(p,g)$.
\end{conj}

\begin{prop}[{\cite[Theorem 1, Corollary 1]{Zilber02esesc}}] \ \\
  (CIT) $\implies$ (CIT${}_{param}$)
\end{prop}
It is not clear that the converse implication holds.

A weak version of the CIT (with parameters), not giving the natural
number $t$, is a theorem. Proofs are given in \cite{Zilber02esesc} and
\cite{DPhil}.
\begin{theorem}[Weak CIT]
  Let $P$ be a complex algebraic variety and let $(V_p)_{p \in P}$ be
  a parametric family of subvarieties of $G = \gm^n(\C)$, the family
  defined over $\bar{\Q}$.

  Then there is a finite family $\HV$ of proper algebraic subgroups of
  $G$ such that, for any coset $\kappa = a \oplus T$ of any algebraic
  subgroup $T$ of $G$ and any $p \in P$, if $X$ is an atypical
  component of the intersection, then there is $H \in \HV$ and $c \in
  G$ such that $X \subs H \cdot c$.
\end{theorem}

\begin{prop}[{\cite[Proposition 5]{Zilber02esesc}}]
  (SC) $\&$ (CIT) $\implies$ (USC)
\end{prop}

\section{Statements for other algebraic groups}

Schanuel's conjecture is a statement about the exponential map of
$\gm$. Since we have several variables, it is more properly seen as a
statement about the exponential maps of the \emph{algebraic tori}, the
algebraic groups of the form $\gm^n$. We can replace the torus by any
(complex) commutative algebraic group, and still have an exponential
map which is a group homomorphism.

Let $S$ be a complex commutative algebraic group, $n$ its dimension,
let $LS \iso \ga^n$ be the tangent space at the identity (the Lie
algebra) and $TS = LS \cross S$ the tangent bundle. The exponential
map $\exp_S$ is a group homomorphism $\ra{LS}{\exp_S}{S}$, and so its
graph $\graph_S$ is an analytic subgroup of $TS$.

In characteristic zero, there is a fairly simple description of
commutative algebraic groups. Any connected such group $S$ has a
Chevalley decomposition, a short exact sequence
\[0 \to \ga^r \cross \gm^l \into S \onto A \to 0\]
where $r,l \in \N$ and $A$ is an abelian variety.

\begin{itemize}
\item Vector groups:  $l=0$, $A = 0$
\item Tori: $r = 0$, $A = 0$
\item Elliptic curves: $r=0$, $l=0$, $\dim A = 1$
\item Abelian varieties: $r=0$, $l=0$
\item Semiabelian varieties: $r=0$
\item Groups with no vector quotient (nvq-groups): the extension by
  $\ga^r$ does not split. In this case, $r \le \dim A$.
\end{itemize}

There are inclusions \\
Elliptic curves $\subs$ Abelian varieties $\subs$ Semiabelian
varieties $\subs$ nvq-groups \\
and also Tori $\subs$ Semiabelian varieties.
However, the intersection between Tori and Abelian varieties is only
the trivial group, and likewise the intersection between Vector groups
and nvq-groups.

If $S$ is a vector group, $LS = S$ and the exponential map is the
identity map on $S$. Thus there is no interesting statement for vector
groups.

The group $S$ may not be defined over $\Q$, and the arithmetic of such
groups is more complicated than that of $\gm$ involving periods,
quasiperiods and so on, so formulating a ``correct'' equivalent of
Schanuel's conjecture is much more difficult. 

Bertolin \cite{Bertolin02} gives a version where $S$ is a product of a torus and
elliptic curves, and shows that it follows from a conjecture on the
periods of 1-motives by Grothendieck and Andr\'e.

The parametric versions are easier to adapt. We assume that $S$ has no
vector quotients.

\begin{conj}[(SC${}_{param}^{nvq}$)]
  Let $V \subs TS$ be an algebraic subvariety with $\dim V < \dim S$.
  Then there is a finitely generated subfield $k$ of $\C$ such that if
  $(a,\exp_S(a)) \in V$ then there is a proper algebraic subgroup $H$
  of $S$ and $c \in LS(k)$ such that $a \in c \cdot LH$.
\end{conj}
A statement of the full conjecture would at least involve specifying
the field $k$ in the case where $V$ is defined over $\Q$.

\begin{conj}[(USC${}_{param}^{nvq}$)]
  Let $V \subs TS$ be an algebraic subvariety with $\dim V < \dim
  S$. There is a finite collection $\mathcal{K}_V$ of cosets of the
  form $c \cdot TH$ where $H$ is a proper algebraic subgroup of $S$
  such that $\graph_S \cap V \subs \bigcup \mathcal{K}_V$.
\end{conj}

The analogues of (SC${}_D$) and (USC${}_D$) hold here too. We give the
stronger statement.

As before, let $F$ be a field of characteristic zero, $\Delta$ a set
of derivations on $F$, and $C$ the intersection of their constant
subfields. For each $D \in \Delta$, let $\logD_S$ be the logarithmic
derivative of $S$ with respect to $D$, and similarly $\logD_{LS}$. Let 
\[\Gamma_{S,\Delta} = \class{(x,y) \in (LS \cross S)(F)}{\logD_S(y) =
  \logD_{LS}(x) \mbox{ for all }D \in \Delta}.\] 
Let $S$ be an nvq-group defined over $C$.

\begin{theorem}[(USC${}_D^{nvq}$)]
  For each parametric family $(V_c)_{c \in P(C)}$ of subvarieties of
  $TS$, with $V_c$ defined over $\Q(c)$, there is a finite set $\HV$
  of proper algebraic subgroups of $S$ such that for each $c \in P(C)$
  and each $(x,y) \in \Gamma^n \cap V_c$, if $\dim V_c - \rk\Jac(x,y)
  < \dim S$, then there is $\gamma \in TS(C)$ and $H \in \HV$ such
  that $(x,y)$ lies in the coset $\gamma \cdot TH$.
\end{theorem}
When $S$ is semiabelian, this is proved in \cite{TEDESV}. The proof
there also adapts to the nvq-group case, which was first proved by
Daniel Bertrand.

As in the torus case, there is the following analytic corollary. As
with CIT, it is a statement about atypical intersections.
\begin{theorem}[{\cite[Theorem 8.1]{DPhil}}]
  Let $P$ be an algebraic variety and $(V_c)_{c \in P(\C)}$ be a
  parametric family of algebraic subvarieties of $TS$. There is a
  finite collection $\HV$ of proper algebraic subgroups of $S$ with
  the following property:

  If $c \in P(\C)$ and $W$ is a connected component of the analytic
  variety $\graph_S \cap V_c$ with analytic dimension $\dim W$
  satisfying $\dim W + \dim S - \dim V_c = t > 0$, then there is $H
  \in \HV$ of codimension at least $t$ and $g \in TS(\C)$ such that
  $W$ is contained in the coset $g \cdot TH$.
\end{theorem}

James Ax proved a version of (SC${}_D$) in the context of arbitrary complex algebraic groups in \cite{Ax72a}. This theorem and the usual exponential (SC${}_D$) are currently known as Ax-Schanuel theorems.

There are many possible analogues of the ``raising to powers''
statements. Some differential fields versions have been proved. We do
not list them here.

The CIT naturally extends to semiabelian varieties, and was given in this context in \cite{Zilber02esesc}. The weak version, below, is a theorem proved in \cite{TEDESV}.

\begin{theorem}[Weak semiabelian CIT]
  Let $S$ be a complex semiabelian variety, $P$ be a complex algebraic
  variety and let $(V_p)_{p \in P}$ be a parametric family of
  subvarieties of $S$, the family defined over $\bar{\Q}$.

  Then there is a finite family $\HV$ of proper algebraic subgroups of
  $S$ such that, for any coset $\kappa = a \oplus T$ of any algebraic
  subgroup $T$ of $S$ and any $p \in P$, if $X$ is an atypical
  component of the intersection, then there is $H \in \HV$ and $c \in
  S$ such that $X \subs c \cdot H$.
\end{theorem}

%\bibliographystyle{alpha}
%\bibliography{../papers,../drafts,../books}

\end{document}